\documentclass[oneside,11pt,reqno]{amsart}
\usepackage[utf8]{inputenc}
\usepackage{stmaryrd}
\usepackage{bbm}
\usepackage{mathrsfs}
\usepackage{enumerate}
\usepackage{color}
\usepackage{latexsym,amsxtra}
\usepackage[dvips]{graphicx}
\usepackage{dsfont}
\usepackage{slashed}
\usepackage[all]{xy}
\usepackage{amscd,graphics}
\usepackage{amsmath,amsfonts,amsthm,amssymb}
\usepackage{latexsym,amsmath}
\usepackage{graphicx,psfrag}
\usepackage{mathabx}
\usepackage{hyperref}
\usepackage{fancyhdr}

\newtheorem{thm}{Theorem}[section]

\newtheorem{fact}[thm]{Fact}

\newtheorem{assum}[thm]{Assumption}
\newtheorem{lem}[thm]{Lemma}
\newtheorem{cor}[thm]{Corollary}
\newtheorem{pro}[thm]{Proposition}
\newtheorem{ques}[thm]{Question}
\theoremstyle{definition}
\newtheorem{defi}[thm]{Definition}
\newtheorem{ex}[thm]{Example}
\theoremstyle{remark}
\newtheorem{rmk}{Remark}

\newcommand{\Spin}{\rm {Spin}}
\newcommand{\Pin}{\rm {Pin}(2)}
\newcommand{\Res}{\rm {Res}}
\newcommand{\Tr}{\rm {Tr}}
\newcommand{\Fr}{\rm {Fr}^{v}}
\newcommand{\SO}{\rm{SO}}
\newcommand{\BF}{\rm{BF}}
\title{ISOTOPY OF THE DEHN TWIST ON $K3\#K3$ AFTER A SINGLE STABILIZATION}
\author{Jianfeng Lin}
\address{Yau Mathematical Sciences Center\newline\indent Tsinghua University \newline\indent Beijing, China}
\email{\rm{linjian5477@gmail.com}}
\date{}

\begin{document}
\maketitle
\begin{abstract}
Kronheimer-Mrowka recently proved that the Dehn twist along a 3-sphere in the neck of $K3\#K3$ is not smoothly isotopic to the identity. This provides a new example of self-diffeomorphisms on 4-manifolds that are isotopic to the identity in the topological category but not smoothly so. (The first such examples were given by Ruberman.) In this paper, we use the $\Pin$-equivariant Bauer-Furuta invariant to show that this Dehn twist is not smoothly isotopic to the identity even after a single stabilization (connected summing with the identity map on $S^{2}\times S^{2}$). This gives the first example of exotic phenomena on simply connected smooth 4-manifolds that do not disappear after a single stabilization.

\end{abstract}
\section{Introduction}
Understanding smooth structures on 4-manifolds remains one of the most difficult topics in low dimensional topology. In this dimension, many results that hold in the topological category will not hold in the smooth category. Such phenomena are called ``exotic phenomena.'' To motivate our discussion, we list three major instances of exotic phenomena:  
\begin{itemize}
    \item By the groundbreaking work of Donaldson \cite{Donaldson87,Donaldson-Kronheimer90} and Freedman \cite{Freedman82} (and many subsequent works), there exist many pairs of simply-connected, closed, smooth 4-manifolds that are homeomorphic but not diffeomorphic.
   
    \item Ruberman \cite{RubermanObstruction} gave the first example of self-diffeomorphisms on 4-manifolds that are isotopic to the identity in the topological category but not smoothly so. Further examples are given by Auckly-Kim-Melvin-Ruberman\cite{AucklyKimMelvinRuberman}, Akbulut \cite{Akbulut1}, Baraglia-Konno \cite{BaragliaKonno} and Kronheimer-Mrowka \cite{KronheimerMrowkaDehnTwist}.
     \item By the combined work of Wall \cite{Wall:Diffeomorphism}, Perron \cite{Perron86}, Quinn \cite{QuinnIsotopy} and Donaldson \cite{Donaldson87}, there exist pairs of embedded 2-spheres in 4-manifolds with simply-connected complement that are topologically isotopic to each other but not smoothly so. (See \cite{Akbulut1,AucklyKimMelvinRuberman} for explicit families of such examples.)
\end{itemize} 
Exotic phenomena appear in all these three problems, which we call as the ``diffeomorphism existence problem'', the ``diffeomorphism isotopy problem'' and the ``surface isotopy problem''. A fundamental principle, as discovered by Wall \cite{Wall:OnSimplyConnected,Wall:Diffeomorphism} in the 1960s, states that  these exotic phenomena will eventually disappear after sufficient many times of stablizations on the 4-manifolds. (Here stabilization means taking connected sum with  $S^{2}\times S^{2}$.) More precisely, now we know the following results:
\begin{itemize}
    \item Wall \cite{Wall:OnSimplyConnected} proved that any pair of homotopy equivalent, simply connected smooth 4-manifolds are stably diffeomorphic.  Namely, they become diffeomorphic after sufficiently many stabilizations. 
    \item Gompf \cite{Gompfstabilization} and Kreck \cite{Krecksurgery} further proved that any pair of homeomorphic orientable smooth 4-manifolds (not necessarily simply connected) are stable diffeomorphic. They also proved that non-orientable pairs can be made stably diffeomorphic by first doing a twisted stabilization (i.e., connected summing a twisted bundle $S^{2}\tilde{\times}S^{2}$). In fact, for any $G$ with $H^{1}(G;\mathbb{Z}/2)\neq 0$, Kreck \cite{Kreck} constructed examples of homeomorphic non-orientable smooth 4-manifold pairs with fundamental group $G$ which are not stably diffeomorphic. (Different constructions of such examples were given by Cappell-Shaneson \cite{Cappell-Shaneson} for $G=\mathbb{Z}/2$ and Akbulut \cite{Akbulut2} for $G=\mathbb{Z}$.) This implies that a twisted stabilization is indeed necessary in the non-orientable case.  
    \item By combining the results of Kreck \cite{Kreckisotopy} and Quinn \cite{QuinnIsotopy}, we know that homotopic diffeomorphisms of any simply-connected smooth 4-manifold are smoothly isotopic after sufficient many stablizations. Here stabilization means first isotoping the diffeomorphisms so that they all fix a small ball $B$ pointwisely and then taking connected sum with the identity map on $S^{2}\times S^{2}$ along $B$.
    \item The work of Wall \cite{Wall:Diffeomorphism}, Perron \cite{Perron86} and Quinn \cite{QuinnIsotopy} shows that for any pair of homologous closed surfaces of the same genus embedded in a 4-manifold with simply-connected complement, they become smoothly isotopic after sufficiently many times of \emph{external stabilizations}. Here external means that the connected sums with $S^{2}\times S^{2}$ are taken away from the surfaces.
\end{itemize}

These results motivate the following natural question:
\begin{ques}
How many stablizations are necessary in each of these three problems?
\end{ques}
There has been a speculation that one stabilization is actually enough in all three problems. This is based on several known results: 
\begin{itemize}
    \item It is shown by Baykur-Sunukjian \cite{BaykurSunukjian} that exotic pairs of nonspin 4-manifolds produced by `standard methods' (logarithmic transforms, knot surgeries, and rational blow downs) all become diffeomorphic after a single stabilization. 
    \item In the large families of examples (of embedded surfaces and self-diffeomorphisms) established in \cite{Akbulut1,AucklyKimMelvinRuberman}, exactly one stabilization is needed.
    \item Auckly-Kim-Ruberman-Melvin-Schwartz \cite{AucklyKimMelvinRubermanSchwartz} proved that any two homologous surfaces $F_{1}, F_{2}$ of the same genus embedded in a smooth 4-manifold $X$ with simply connected complements are smoothly isotopic after a single stabilization if they are not characteristic (i.e. $[F_{i}]$ is not dual to the Stiefel-Whitney class $w_{2}(X)$). This shows that in the non-characteristic case, one stabilization is indeed enough in the surface isotopy problem. (When the surfaces are characteristic, they proved a similar result involving a single twisted stabilization.) 
\end{itemize}

In this paper, we prove the the following theorem. 
\begin{thm}[Main Theorem]\label{thm: main} Let $\delta$ be the Dehn twist along a separating 3-sphere in the neck of the connected sum $K3\#K3$. Then $\delta$ is \textbf{not} smoothly isotopic to the identity map even after a single stabilization.
\end{thm}
To the authors’ knowledge, Theorem \ref{thm: main} provides the first example that exotic phenomena on simply connected smooth 4-manifolds do not disappear after a single stabilization with respect to $S^{2}\times S^{2}$. In particular, it implies that one stabilization is in general \emph{not} enough in the diffeomorphism isotopy problem.

Note that Kronheimer-Mrowka \cite{KronheimerMrowkaDehnTwist} proved that $\delta$ itself is not smoothly isotopic to the identity, using the nonequivariant Bauer-Furuta invariant for spin families. Our result is based on Kronheimer-Mrowka's theorem and makes use of the $\Pin$-equivariant version of the Bauer-Furuta invariant. This invariant was defined in \cite{BauerFuruta1} (for a single manifold) and in \cite{Xu,Szymik} (for families). It has been extensively studied in many papers including \cite{baraglia2019constraints,baraglia2019bauer} and it is the central tool in Furuta's proof of the $\frac{10}{8}$-theorem \cite{Furuta10/8}. The idea of using gauge-theoretic invariant for families to study isotopy problem first appears in \cite{RubermanObstruction}. The idea of using $\Pin$-equivariant Bauer-Furuta invariant to further study Dehn twists on 4-manifolds was suggested by Kronheimer-Mrowka in \cite{KronheimerMrowkaDehnTwist}.

We outline the proof of Theorem \ref{thm: main} as follows:  By taking the mapping torus of $\delta$, we form a smooth bundle $N$ with fiber $K3\#K3$ and base $S^{1}$. Then it suffices to show that the bundle $\tilde{N}$, formed by fiberwise connected sum  between $N$ and $(S^{2}\times S^{2})\times S^{1}$, is not a product bundle. This is proved by showing that the $\Pin$-equivariant Bauer-Furuta invariant  $\BF^{\Pin}(\tilde{N})$ is nonvanishing for both spin structures. Note that $\BF^{\Pin}(\tilde{N})$ equals the product of $\BF^{\Pin}(N)$ with the Euler class $e_{\tilde{\mathbb{R}}}$ (a stable homotopy class represented by the inclusion from $S^{0}=\{0,\infty\}$ to the 1-dimensional representation sphere $S^{\tilde{\mathbb{R}}}$). We prove by contradiction and assume that  
\begin{equation}\label{eq: vanishing BF}
\BF^{\Pin}(N)\cdot e_{\tilde{\mathbb{R}}}=0   
\end{equation}
 This gives information on $\BF^{\Pin}(N)$ and its $S^{1}$-reduction $\BF^{S^{1}}(N)\in \{S^{\mathbb{R}+2\mathbb{H}},S^{6\tilde{\mathbb{R}}}\}^{S^{1}}$. We can explicitly compute the homotopy group $\{S^{\mathbb{R}+2\mathbb{H}},S^{6\tilde{\mathbb{R}}}\}^{S^{1}}$ as $\mathbb{Z}\oplus \mathbb{Z}/2$. Based on this computation, information from (\ref{eq: vanishing BF}) and the fact that $\BF^{S^{1}}(N)$ gives a vanishing family Seiberg-Witten invariant, we can prove that $\BF^{S^{1}}(N)=0$. This further implies that the nonequivariant Bauer-Furuta invariant $\BF^{\{e\}}(N)$ is vanishing, which contradicts with Kronheimer-Mrowka's result that $\BF^{\{e\}}(N)$ equals the nonzero element $\eta^{3}\in \pi_{3}$. Note that $e_{\tilde{\mathbb{R}}}$ becomes trivial when reducing to the subgroup $S^{1}\subset \Pin$. As a consequence, the $S^{1}$-equivariant Bauer-Furuta invariant vanishes after a single stabilization (just like the classical Seiberg-Witten invariants and Donaldson's polynomial invariants). This explains why the $\Pin$-equivariance is essential in our proof.

We end this introductory section by remarking that it is still open whether one stabilization is enough to make any pairs of simply connected, heomeomorphic 4-manifold diffeomorphic. (See \cite{Akbulut-Mrowka-Ruan,Donaldsonorientation,Fintushel-Stern} for a possible approach using the 2-torsion instanton invariants.) It's also unknown whether two homotopic characteristic surfaces with simply connected complements become smoothly isotopic after a single stabilization.  The proof of Theorem \ref{thm: main} suggests that Bauer-Furuta invariant could be useful in attacking these problems. As a first step, one needs to establish new examples of spin 4-manifolds with sufficiently interesting higher-dimensional $\Pin$-equivariant Bauer-Furuta invariants.  Note that in a recent paper by the author and Mukherjee \cite{LinMukherjee}, we use Theorem \ref{thm: main} to establish the first pair of orientable exotic surfaces (in a punctured K3 surface) which are not smoothly isotopic even after one stabilization.

The paper is organized as follows: In section 2, we give a brief review on some basic $\Pin$-equivariant stable homotopy theory and recall definition of the equivariant Bauer-Furuta invariant. We also use this section to set up notations and to adapt some standard results to the setting we need. The actual proof of Theorem \ref{thm: main} is given in Section 3. Experts may directly skip to Section 3 and occasionally refer back to Section 2 for notations and results.
\\
\\
\textbf{Acknowledgement:} This author is partially supported by NSF Grant DMS-1949209. The author would like to thank Tye Lidman and Danny Ruberman for very enlightening conversations, thank Mark Powell for pointing out Kreck's work \cite{Kreck} and thank Selman Akbulut for explaining his work in \cite{Akbulut2, Akbulut1}.

\section{Background Knowledge }
\subsection{Background on Pin(2)-equivariant Homotopy Theory}
In this section, we collect some standard results (mostly from \cite{AdamsPrerequisites,Schwede,May,LewisMaySteinberher}) on $G$-equivariant stable homotopy theory in the case 
$$
G=\Pin=\{e^{i\theta}\}\cup \{j\cdot e^{i\theta}\}\subset \mathbb{H}.
$$
Instead of stating the most general form of these results, we will only focus on the special cases that are actually needed in our argument. We refer to \cite{AdamsPrerequisites,Schwede} for an introduction to the equivariant stable homotopy theory (in the case of finite group) and to \cite{May,LewisMaySteinberher} for a more general treatment. 

Since all objects we study here are finite $G$-CW complexes, for simplicity, we will work with the $G$-equivariant Spanier-Whitehead category \cite{AdamsPrerequisites} (instead of the homotopy category of $G$-spectra). Of course, there are a lot of drawbacks (e.g. one can not always take limits/colimits). But it is enough for our purpose. 

\subsubsection{Basic facts and definitions}

Let $U$ be an countably infinite dimensional $G$-representation space equipped with a $G$-invariant inner product, which we call a ``universe''. We assume that $U$ contains the following concrete representation
$$
(\mathop{\oplus}\limits_{\infty}\mathbb{R})\bigoplus (\mathop{\oplus}\limits_{\infty}\tilde{\mathbb{R}})\bigoplus (\mathop{\oplus}\limits_{\infty}\mathbb{H}).
$$
Here $\mathbb{R}$ is the trivial representation; $\tilde{\mathbb{R}}$ is the 1-dimensional representation on which $S^{1}$ acts trivially and $j$ acts as $-1$; and $\mathbb{H}$ is acted upon by $G$ via the left multiplication in quaternion.  

To apply the results in \cite{May} directly without checking further conditions, we further assume that $U$ is ``complete''. This means that $U$ contains infinitely many copies of all isomorphism classes of irreducible $G$-representations. \footnote{Since all $G$-CW complexes we consider in this paper can have only $G, S^{1}$ or $\{e\}$ as the isotropy group, all argument we make actually will still hold for the incomplete universe $
(\mathop{\oplus}\limits_{\infty}\mathbb{R})\bigoplus (\mathop{\oplus}\limits_{\infty}\tilde{\mathbb{R}})\bigoplus (\mathop{\oplus}\limits_{\infty}\mathbb{H})
$, which is more relevant to the geometric setting.}

We will use $H$ to denote either the group $G$ or its subgroups $S^{1}$ or $\{e\}$. By restricting the $G$-action on $U$, we can also use $U$ as a complete $H$-universe. We use $\mathbf{R}_{H}$ to denote the set of all \emph{finite dimensional} $H$-representations contained in $U$. We will treat $\mathbf{R}_{G}$ as a subset of $\mathbf{R}_{S^{1}}$ and $\mathbf{R}_{\{e\}}$ by restricting the $G$-action.

For any $V\in \mathbf{R}_{H}$, we use $S^{V}$ to denote the 1-point compactification of $V$ (called the representation sphere) and use $S(V)$ to denote the unit sphere. We set $\infty$ as the base point of $S^{V}$ and we use $S(V)_{+}$ to denote the union of $S(V)$ with a disjoint base point.

Let $X,Y,Z$ be based finite $H$-CW complexes (see for example \cite[Chapter I]{tom2011transformation} for definition). We use the notation $[X,Y]^{H}$ to denote the set of homotopy classes of based $H$-maps from $X$ to $Y$ (i.e. maps that preserve the base point and are equivariant under $H$). 

Given any $V,W\in \mathbf{R}_{H}$ with $V\subset W$, let $V^{\perp}$ be the orthogonal complement of $V$ in $W$. Then smashing with the identity map on $S^{V^{\perp}}$ provides a map:
$$
[S^{V}\wedge X, S^{V}\wedge Y]^{H}\rightarrow [S^{W}\wedge X,S^{W}\wedge Y]^{H}
$$
One can check that these maps make the collection $$\{[S^{V}\wedge X, S^{V}\wedge Y]^{H}\}_{V\in \mathbf{R}_{H}}$$ into a direct system. We define $\{X,Y\}^{H}$ as the direct limit of this system. 
As in the nonequivariant case, the set $\{X,Y\}^{H}$ is actually an abelian group. A based $H$-map 
$$
S^{V}\wedge X\rightarrow S^{V}\wedge Y \text{ for }V\in \mathbf{R}_{H}
$$
will be called a stable $H$-map from $X$ to $Y$. An element in the group $\{X,Y\}^{H}$ will be called a stable homotopy class of $H$-maps.

\begin{fact}Given any based $H$-map $f:X\rightarrow Y$, we  form the mapping cone $Cf$ and let $i:Y\rightarrow Cf$  be the natural inclusion. Then for any $Z$, the functor $\{*,Z\}^{H}$ is a generalized cohomology theory \cite[Page 157]{May}.  As a result,  there is a long exact sequence 
\begin{equation}\label{long exact sequence}
\cdots \rightarrow\{S^{\mathbb{R}}\wedge X,Z\}^{H}\xrightarrow{\partial}\{Cf,Z\}^{H}\xrightarrow{i^{*}} \{Y,Z\}^{H}\xrightarrow{f^{*}} \{X,Z\}^{H}\xrightarrow{\partial} \{Cf,S^{\mathbb{R}}\wedge  Z\}^{H}\rightarrow \cdots
\end{equation}
associated to the cofiber sequence $X\xrightarrow{f} Y\xrightarrow{i}Cf
$. 
\end{fact}

\begin{fact}\label{Fact: quotient map} Suppose the $H$-action on $X$ is free away from the base point. Then there is a natural map
\begin{equation}\label{eq: quotient iso}
q_{H}: \{X,Y\}^{H}\rightarrow \{X/H,Y/H\}^{\{e\}}
\end{equation}
from the equivariant homotopy group to the nonequivariant homotopy group of the quotient space. This map is constructed as follows: Since the $H$-action on $X$ is free away from the base point, any $[f]\in \{X,Y\}^{H}$ can be represented by an $H$-map
$
f:S^{V}\wedge X\rightarrow S^{V}\wedge Y
$ such that the $H$-action on $V$ is trivial (see \cite[Proposition 5.5]{AdamsPrerequisites} and  \cite[Theorem 2.8 on Page 65]{LewisMaySteinberher}). The map $f$ induces a nonequivariant map between the quotient space.
$$
f/H:S^{V}\wedge (X/H)= (S^{V}\wedge X)/H\rightarrow (S^{V}\wedge Y)/H=S^{V}\wedge (Y/H).
$$
Then we define $q_{H}([f])$ as $[f/H]$. One can check that this does not depend on the choice of $f$ and $V$.
\end{fact}
\begin{fact} (See \label{fact: quotient isomorphism} \cite[Theorem 5.3]{AdamsPrerequisites} and \cite[Theorem 4.5 on Page 78]{LewisMaySteinberher}.) Suppose the $H$-action on $X$ is free away from the base point and the $H$-action on $Y$ is trivial. Then the map  $q_{H}$ is an isomorphism. \end{fact}

For the rest of the section, we assume $X,Y$ are based, finite $G$-CW complexes. The next few facts concern various relations between the $G$-equivariant homotopy groups and the $S^{1}$-equivariant homotopy groups.

\begin{fact} (See \cite[Theorem 5.1]{AdamsPrerequisites}  and \cite[Theorem 4.7 on Page 79]{LewisMaySteinberher}.)
There is a natural isomorphism 
\begin{equation}
\iota:\{X,Y\}^{S^{1}}\xrightarrow{\cong} \{X\wedge (S(\tilde{\mathbb{R}})_{+}),Y\}^{G}  
\end{equation}
constructed as follows: Take any $[f]\in \{X,Y\}^{S^{1}}$ represented by an $S^{1}$-map
$$
f:S^{V}\wedge X\rightarrow S^{V}\wedge Y.
$$
By enlarging $V$ if necessary, we may assume $V\in \mathbf{R}_{G}$. Then we consider the $G$-map
$$
f':S^{V}\wedge X\wedge (S(\tilde{\mathbb{R}})_{+})=((S^{V}\wedge X) \times \{1\} )\vee ((S^{V}\wedge X) \times \{-1\} ) \rightarrow Y
$$
defined by setting 
$$
f'(x\times\{1\})=f(x)\text{ and } f'(x\times \{-1\}):=jf(j^{-1}x)$$ for any $x\in S^{V}\wedge X.
$
We let $\iota([f])=[f']$. This map $\iota$ turns out to be an isomorphism. 
\end{fact}

Next, we recall the two operations about changing groups, namely the restriction map 
\begin{equation}\label{eq: restriction}
\Res^{G}_{S^{1}}: \{X,Y\}^{G}\rightarrow \{X,Y\}^{S^{1}} 
\end{equation}
and the transfer map 
\begin{equation}\label{eq: transfer}
\Tr^{G}_{S^{1}}: \{X,Y\}^{S^{1}}\rightarrow \{X,Y\}^{G}.
\end{equation}
The restriction map is defined by simply ignoring the $j$-action. To define the transfer map, we consider the Pontrjagin-Thom map
$$p:S^{\tilde{\mathbb{R}}}\rightarrow S^{\tilde{\mathbb{R}}}\wedge S(\tilde{\mathbb{R}})_{+}$$ that crushes all points outside a normal neighborhood of   $S(\tilde{\mathbb{R}})$ in $S^{\tilde{\mathbb{R}}}$. (Here we identify the Thom space of 
the normal bundle of $S(\tilde{\mathbb{R}})$ as $S^{\tilde{\mathbb{R}}}\wedge (S(\tilde{\mathbb{R}})_{+})$.) Then the transfer map is defined as the composition
\begin{equation}\label{eq: defi transfer}
\{X,Y\}^{S^{1}}\xrightarrow{\iota} \{(S(\tilde{\mathbb{R}})_{+})\wedge X ,Y\}^{G}=  \{S^{\tilde{\mathbb{R}}}\wedge (S(\tilde{\mathbb{R}})_{+})\wedge X,S^{\tilde{\mathbb{R}}}\wedge Y\}^{G}\xrightarrow{p^{*}} \{S^{\tilde{\mathbb{R}}}\wedge X,S^{\tilde{\mathbb{R}}}\wedge Y\}^{G}=\{X,Y\}^{G} .
\end{equation}
To describe the composition of transfer and restriction, we define the conjugation map
\begin{equation}\label{conjugation map}
c_{j}:\{X,Y\}^{S^{1}}\rightarrow \{X,Y\}^{S^{1}}
\end{equation}
as follows: 
Take any element $[f]\in \{X,Y\}^{S^{1}}$ represented by a $S^{1}$-map $f:S^{V}\wedge X\rightarrow S^{V}\wedge Y$. By enlarging $V$ if necessary, we may assume $V\in \mathbf{R}_{G}$. Then $c_{j}([f])$ is represented by the composition $$
S^{V}\wedge X\xrightarrow{j^{-1}}  S^{V}\wedge X\xrightarrow {f} S^{V}\wedge Y \xrightarrow{j }S^{V}\wedge Y.
$$
Note that when the $S^{1}$-action on $X$ is free away from the base point, the maps $c_{j}$ and the map $q_{S^{1}}$ defined in (\ref{eq: quotient iso}) are compatible. That means 
\begin{equation}\label{quotient conjugation}
q_{S^{1}}(c_{j}(\alpha))= j\circ q_{S^{1}}(\alpha)\circ j^{-1} ,\forall \alpha\in \{X,Y\}^{S^{1}}.
\end{equation}
Here $j$ and $j^{-1}$ are treated as elements in $\{Y/S^{1},Y/S^{1}\}^{\{e\}}$ and  $\{X/S^{1},X/S^{1}\}^{\{e\}}$ respectively.

The following fact is a special case of the double coset formula \cite[Chapter XVIII, Theorem 4.3]{May}. It can be verified directly by unwinding the definitions.

\begin{fact}
For any $\alpha \in\{X,Y\}^{S^{1}}$, one has 
\begin{equation}\label{eq: double coset formula}
    \Res^{G}_{S^{1}}\Tr^{G}_{S^{1}}(\alpha)=\alpha+c_{j}(\alpha)
\end{equation}

\end{fact}

We end this subsection by an alternative description of the image of $\Tr^{G}_{S^{1}}$. 
\begin{lem}\label{lem: image of transfer}
Let $e_{\tilde{\mathbb{R}}}\in \{S^{0}, S^{\tilde{\mathbb{R}}}\}^{G}$ be the element represented by the inclusion map 
\begin{equation}\label{eq: euler class}
S^{0}=\{0,\infty\}\hookrightarrow S^{\tilde{\mathbb{R}}}.
\end{equation}
(This element is called the Euler class of $\tilde{\mathbb{R}}$.) Then the kernel of the map
\begin{equation}\label{eq: multiplication euler}
\{X,Y\}^{G}\xrightarrow{e_{\tilde{\mathbb{R}}}\cdot } \{ X,S^{\tilde{\mathbb{R}}}\wedge Y\}^{G}
\end{equation}
equals the image of the transfer map (\ref{eq: transfer}).
\end{lem}
\begin{proof}
There is a cofiber sequence $S^{0}\hookrightarrow S^{\tilde{\mathbb{R}}}\xrightarrow{p}  S^{\tilde{\mathbb{R}}}\wedge S(\tilde{\mathbb{R}})_{+}
$. Smashing this sequence with $X$ and applying the functor $\{*,S^{\tilde{\mathbb{R}}}\wedge Y\}^{G}$, we get the exact sequence:
$$
\{(S^{\tilde{\mathbb{R}}}\wedge S(\tilde{\mathbb{R}})_{+})\wedge X,S^{\tilde{\mathbb{R}}}\wedge Y\}\xrightarrow {p_{*}} \{S^{\tilde{\mathbb{R}}}\wedge X,S^{\tilde{\mathbb{R}}}\wedge Y\}^{G}\xrightarrow{e_{\tilde{\mathbb{R}}}\cdot} \{ X,S^{\tilde{\mathbb{R}}}\wedge Y\}^{G}
$$
So we see that image of $p_{*}$ equals kernel of the map (\ref{eq: multiplication euler}). The lemma follows from definition of $\Tr^{G}_{S^{1}}$ (see (\ref{eq: defi transfer})).
\end{proof}

\subsubsection{The characteristic homomorphism} In this subsection, we define the characteristic homomorphism 
$$t:\{S^{a\mathbb{R}+b\mathbb{H}},S^{d\tilde{\mathbb{R}}}\}^{S^{1}}\rightarrow \mathbb{Z},
$$ following \cite{BauerFuruta1},
where $a,b,c$ are nonnegative integers with $d\geq a+2$. This homomorphism is of interest to us because the (family) Seiberg-Witten invariant can be obtained by applying $t$ on the Bauer-Furuta invariant. Note that although $\tilde{\mathbb{R}}$ is trivial as an $S^{1}$-representation, we still distinguish it with $\mathbb{R}$ in order to keep track of the $j$-action. 

To define $t$, we take the smash product of the cofiber sequence $$S^{0}\rightarrow S^{b\mathbb{H}}\rightarrow S^{\mathbb{R}}\wedge S(b\mathbb{H})_{+}$$ with the sphere $S^{a\mathbb{R}}$ and get the cofiber sequence 
\begin{equation}\label{eq: cofiber sequence}
S^{a\mathbb{R}}\rightarrow S^{a\mathbb{R}+b\mathbb{H}}\rightarrow S^{(a+1)\mathbb{R}}\wedge S(b\mathbb{H})_{+} .    
\end{equation}
This induces the long exact sequence
\begin{equation}\label{eq: long exact 1}
\cdots \rightarrow\{S^{(a+1)\mathbb{R}},S^{d\tilde{\mathbb{R}}}\}^{S^{1}}\rightarrow \{S^{(a+1)\mathbb{R}}\wedge S(b\mathbb{H})_{+}  ,S^{d\tilde{\mathbb{R}}}\}^{S^{1}}\rightarrow \{S^{a\mathbb{R}+b\mathbb{H}},S^{d\tilde{\mathbb{R}}}\}^{S^{1}}\rightarrow \{S^{a\mathbb{R}},S^{d\tilde{\mathbb{R}}}\}^{S^{1}}\rightarrow \cdots
\end{equation}
Since $d\geq a+2$, the equivariant Hopf theorem \cite[Section 8.4]{dieck2006transformation} states that the stable homotopy class of an $S^{1}$-equivariant stable map from $S^{a\mathbb{R}}$ or $S^{(a+1)\mathbb{R}}$ to $S^{d\tilde{\mathbb{R}}}$ is determined by its mapping degree on the $S^{1}$-fixed point sets. Since this mapping degree is always zero for dimension reason, we get $$\{S^{a\mathbb{R}},S^{d\tilde{\mathbb{R}}}\}^{S^{1}}=\{S^{(a+1)\mathbb{R}},S^{d\tilde{\mathbb{R}}}\}^{S^{1}}=0.$$
Therefore, we get an isomorphism 
\begin{equation}\label{eq: rep sphere to unit sphere}
\xi:\{S^{(a+1)\mathbb{R}}\wedge S(b\mathbb{H})_{+}  ,S^{d\tilde{\mathbb{R}}}\}^{S^{1}}\xrightarrow{\cong} \{S^{a\mathbb{R}+b\mathbb{H}},S^{d\tilde{\mathbb{R}}}\}^{S^{1}}.
\end{equation}
Note that the $S^{1}$-action on $S^{(a+1)\mathbb{R}}\wedge S(b\mathbb{H})_{+}$ is free away from the base point, with quotient space $S^{(a+1)\mathbb{R}}\wedge \mathbb{C}P^{2b-1}_{+}$. By composing  $\xi^{-1}$ with the isomorphism $q_{S^{1}}$ given in (\ref{eq: quotient iso}), we get the following isomorphism 
\begin{equation}
\psi=q_{S^{1}}\circ \xi^{-1}: \{S^{a\mathbb{R}+b\mathbb{H}},S^{d\tilde{\mathbb{R}}}\}^{S^{1}}\xrightarrow{\cong} \{S^{(a+1)\mathbb{R}}\wedge \mathbb{C}P^{2b-1}_{+}, S^{d\tilde{\mathbb{R}}}\}^{\{e\}}
\end{equation}
\begin{defi}
Suppose $d-a$ is an odd number less or equal to $4b-1$. Then we define the characteristic homomorphism 
$$t:\{S^{a\mathbb{R}+b\mathbb{H}},S^{d\tilde{\mathbb{R}}}\}^{S^{1}}\rightarrow \mathbb{Z},
$$ 
by setting $t(\alpha)$ as the image of $1$ under the induced map on the  reduced cohomology 
$$
(\psi(\alpha))^{*}: \mathbb{Z}=\tilde{H}^{d}(S^{d\tilde{\mathbb{R}}})\rightarrow \tilde{H}^{d}(S^{(a+1)\mathbb{R}}\wedge \mathbb{C}P^{2b-1}_{+})\cong \mathbb{Z}.$$
Here we use the standard orientations on $S^{d\tilde{\mathbb{R}}}$, $S^{(a+1)\mathbb{R}}$ and $\mathbb{C}P^{\frac{d-a-1}{2}} $ to identify the homology groups as $\mathbb{Z}$. If  either $d-a$ is even or $d-a>4b-1$, we simply define $t$ as the zero map. 
\end{defi}

To discuss the behavior of $t$ under the conjugation map $c_{j}$ defined in (\ref{conjugation map}), we prove the following lemma.
\begin{lem}\label{lem: tau conjugation}
For any $\alpha\in  \{S^{a\mathbb{R}+b\mathbb{H}},S^{d\tilde{\mathbb{R}}}\}^{S^{1}}$, we have $$\psi(c_{j}(\alpha))=(-1)^{d} m\circ \psi(\alpha)$$ where $m\in  \{\mathbb{C}P^{2b-1}_{+},\mathbb{C}P^{2b-1}_{+} \}^{\{e\}}$ is the ``mirror reflection map'' defined as
$$
m([z_{1},z_{2},z_{3},z_{4},\cdots, z_{2b-1},z_{2b}])=([-\bar{z}_{2},\bar{z}_{1},-\bar{z}_{4},\bar{z}_{3},\cdots, -\bar{z}_{2b},\bar{z}_{2b-1}]) \text{ for }z_{i}\in \mathbb{C}.
$$
\end{lem}
\begin{proof}
By formula (\ref{quotient conjugation}),  $\psi(c_{j}(\alpha))$ equals the composition of $\psi(\alpha)$ with the elements $$j\in \{S^{d\tilde{\mathbb{R}}},S^{d\tilde{\mathbb{R}}}\}^{\{e\}}$$
and 
$$j^{-1}\in \{S^{(a+1)\mathbb{R}}\wedge \mathbb{C}P^{2b-1}_{+},S^{(a+1)\mathbb{R}}\wedge \mathbb{C}P^{2b-1}_{+}\}^{\{e\}},$$ which are just $(-1)^{d}$ and a suspension of $m$ respectively.
\end{proof}
\begin{cor}\label{cor: conjugate char homo}
When $d-a$ is odd, we have $t(c_{j}(\alpha))=(-1)^{\frac{3d-a-1}{2}}t(\alpha)$ for any $\alpha$.
\end{cor}
\begin{proof}
When restricted to $\mathbb{C}P^{1}$, the map $m$ is just the antipodal map so has degree $-1$. Using the ring structure on $H^{*}(\mathbb{C}P^{2b-1})$, we see that $m$ has degree $(-1)^{\frac{d-a-1}{2}}$ on $\tilde{H}^{d}(S^{(a+1)\mathbb{R}}\wedge \mathbb{C}P^{2b-1}_{+})$. Now the result follows from Lemma \ref{lem: tau conjugation}. 
\end{proof}
We end this section by the following result, which is essentially the algebraic version of the vanishing result for Seiberg-Witten invariant of connected sums. 
\begin{lem}\label{lem: SW vanishing}
Given any $\alpha_{1}\in \{S^{a_{1}\mathbb{R}+b_{1}\mathbb{H}},S^{d_{1}\tilde{\mathbb{R}}}\}^{S^{1}}$ and $\alpha_{2}\in \{S^{a_{2}\mathbb{R}+b_{2}\mathbb{H}},S^{d_{2}\tilde{\mathbb{R}}}\}^{S^{1}}$, we have $t(\alpha_{1}\alpha_{2})=0$ if $d_{1}>a_{1}$ and $d_{2}>a_{2}$.
\end{lem}
\begin{proof}
The product $\alpha_{1}\alpha_{2}$ belongs to the group $$\{S^{(a_{1}+a_{2})\mathbb{R}+(b_{1}+b_{2})\mathbb{H}},S^{(d_{1}+d_{2})\tilde{\mathbb{R}}}\}^{S^{1}}.$$ Therefore, $t(\alpha_{1}\alpha_{2})$ can possibly be nonzero only if $d_{1}+d_{2}-a_{1}-a_{2}$ is odd. Without loss of generality, we may assume $d_{1}-a_{1}$ is odd and $d_{2}-a_{2}$ is even. 
Since $d_{i}>a_{i}$ for $i=1,2$, the group $\{S^{a_{i}\mathbb{R}},S^{d_{i}\tilde{\mathbb{R}}}\}^{S^{1}}$ is vanishing. By the long exact sequence (\ref{eq: long exact 1}), we see that $\alpha_{i}$ equals the image of some element $$\beta_{i}\in \{S^{(a_{i}+1)\mathbb{R}}\wedge S(b_{i}\mathbb{H})_{+}  ,S^{d_{i}\tilde{\mathbb{R}}}\}^{S^{1}}=\{S^{a_{i}\mathbb{R}}\wedge (S^{b_{i}\mathbb{H}}/S^{0})  ,S^{d_{i}\tilde{\mathbb{R}}} \}^{S^{1}}.$$
Here we identify $S^{b_{i}\mathbb{H}}/S^{0}$ with $S^{\mathbb{R}}\wedge S(b_{i}\mathbb{H})_{+}$ by treating $S^{\mathbb{R}}$ as the one-point compactification of $(0,+\infty)$ and sending $v\in \mathbb{H}^{b_{i}}\setminus \{0\}$ to $(|v|,\frac{v}{|v|})\in (0,\infty)\times S(b_{i}\mathbb{H})$. 

Next, we consider the following commutative diagram
\begin{equation}\label{diagram: smash and quotient}
\xymatrix{
S^{(b_{1}+b_{2})\mathbb{H}}\ar[r]^-{q}\ar[d]^-{\cong} &S^{(b_{1}+b_{2})\mathbb{H}}/S^{0}\ar[d]^{\gamma}\\
S^{b_{1}\mathbb{H}}\wedge S^{b_{2}\mathbb{H}}\ar[r]^-{q_{1}\wedge q_{2}}&  (S^{b_{1}\mathbb{H}}/S^{0})\wedge (S^{b_{2}\mathbb{H}}/S^{0}) \\
}
\end{equation}
where $q,q_{1},q_{2}$ and  
$$
\gamma: S^{(b_{1}+b_{2})\mathbb{H}}/S^{0}\rightarrow S^{(b_{1}+b_{2})\mathbb{H}}/((S^{0}\wedge S^{b_{2}\mathbb{H}})\cup (S^{b_{1}\mathbb{H}}\wedge S^{0}))= (S^{b_{1}\mathbb{H}}/S^{0})\wedge (S^{b_{2}\mathbb{H}}/S^{0})
$$
are all quotient maps. From diagram (\ref{diagram: smash and quotient}), we see that 
$$
\alpha_{1}\wedge \alpha_{2}=(\beta_{1}\wedge  \beta_{2})\circ (q_{1} \wedge q_{2}) = (\beta_{1}\wedge  \beta_{2})\circ \gamma\circ q.
$$
Therefore, we have $\xi(\alpha_{1}\alpha_{2})=(\beta_{1}\beta_{2})\circ \gamma$.

Moreover, checking the explicit construction of the map $q_{S^{1}}$ given in Fact \ref{Fact: quotient map}, we see that $q_{S^{1}}$ is also natural under the smash product and composition. Therefore, we have $$\psi(\alpha_{1}\alpha_{2})=q_{S^{1}}(\xi(\alpha_{1}\alpha_{2}))=q_{S^{1}}(\beta_{1}\beta_{2})\circ q_{S^{1}}(\gamma)$$
and $q_{S^{1}}(\beta_{1}\beta_{2})$ equals the composition
\begin{equation*}
\begin{split}
S^{(a_{1}+a_{2}+2)\mathbb{R}}\wedge ((S(b_{1}\mathbb{H})_{+}\wedge S(b_{2}\mathbb{H})_{+})/S^{1})\rightarrow \\ (S^{(a_{1}+1)\mathbb{R}}\wedge (S(b_{1}\mathbb{H})_{+})/S^{1}))\wedge (S^{(a_{2}+1)\mathbb{R}}\wedge (S(b_{2}\mathbb{H})_{+})/S^{1}))
\xrightarrow{q_{S^{1}}(\beta_1)\wedge q_{S^{1}}(\beta_{2})} S^{d_{1}\tilde{\mathbb{R}}}\wedge S^{d_{2}\tilde{\mathbb{R}}}
\end{split}
\end{equation*}
Because $d_{2}-a_{2}$ is even, the cohomology $\tilde{H}^{d_{2}}(S^{(a_{2}+1)\mathbb{R}}\wedge (S(b_{2}\mathbb{H})_{+})/S^{1}))=0$. So $q_{S^{1}}(\beta_{2})$ induces trivial map on the reduced cohomology. This implies that $\psi(\alpha_{2}\alpha_{2})$ induces trivial map on $\tilde{H}^{d_{1}+d_{2}}(*)$. Hence we have $t(\alpha_{1}\alpha_{2})=0$.
\end{proof}
\subsection{The Pin(2)-equivariant Bauer-Furuta Invariant for Spin Families}
In this section, we briefly summarize the definition and some important properties of the Bauer-Furuta invariant for spin families. This invariant was originally defined in \cite{BauerFuruta1} for a single 4-manifold. The family version was first defined in \cite{Xu} and \cite{Szymik} and later extensively studied in \cite{baraglia2019bauer,baraglia2019constraints}. Because we want to construct the Bauer-Furuta invariant as a concrete element in the $G$-equivariant stable homotopy group of spheres, some care must be taken in the construction. 
\subsubsection{Spin structures on circle family of 4-manifolds}
Let $N$ be a smooth fiber bundle whose fiber is a closed spin $4$-manifold $M$ and whose base is another closed manifold $B$. For simplicity, we will make the following assumption throughout the paper:
\begin{assum}\label{assum: bundle}The bundle $N$ satisfies the following property: 
\begin{enumerate}
    \item $M$ is simply connected;
    \item The signature $\sigma(M)\leq 0$;
    \item Let $M_{x}$ be the fiber over the point $x\in B$. Then the action of $\pi_{1}(B,x)$ on $H^{2}(M_x;\mathbb{Z})$ (given by the holonomy of the bundle) is trivial.
\end{enumerate}
\end{assum}
We equip $N$ with a Riemannian metric and let $\Fr(N)$ be the frame bundle of the \emph{vertical} tangent bundle of $N$. This is an $\SO(4)$-bundle over $N$. 

\begin{defi}
A spin structure $\mathfrak{s}$ on $N$ is a double covering map $\pi:P\rightarrow\Fr(N)$ that restricts to a nontrivial covering map $\Spin(4)\rightarrow \SO(4)$ on each fiber. Two spin structures $(\pi,P)$ and $(\pi',P')$ are called isomorphic if there exists homeomorphism $P\rightarrow P'$ that covers the \emph{identity map on $\Fr(N)$.}  
\end{defi}
\begin{defi}
The pair $(N,\mathfrak{s})$ is called a spin family. Two spin families $(N_{1},\mathfrak{s}_{1})$ and  $(N_{2},\mathfrak{s}_{2})$ over the same base $B$ are called ``isomorphic'' if there exists a bundle isomorphism $f:N_{1}\rightarrow N_{2}$ such that $f^{*}(\mathfrak{s}_{2})$ is isomorphic to $\mathfrak{s}_{1}$.
\end{defi}

We are mainly interested in the case that $B$ is a circle or a point. By our Assumption \ref{assum: bundle}, $N$ has a unique spin structure when $B$ is a point and has two spin structures when $B$ is a circle. We give an explicit description of these two spin structures as follows: Let $\pi_{M}:P_{M}\rightarrow \rm{Fr}(M)$ be the covering map given by the unique spin structure on $M$. Then the bundle $N$ is obtained by gluing the two boundary components of $M\times [0,1]$ via a diffeomorphism $f:M\rightarrow M$. The diffeomorphism induces a map $f_{*}: \rm{Fr}(M)\rightarrow \rm{Fr}(M)$, which has two lifts $f^{\pm}_{*}:P_{M}\rightarrow P_{M}$. These lifts differ from each other by the deck transformation $\tau: P_{M}\rightarrow P_{M}$. We use $f_{*}^{\pm}$ to glue the two boundary components of $P_{M}\times I$ and form two spin structures on $N$. 

\begin{defi}\label{defi: product family and twisted family}
When $N=M\times S^{1}$, the maps $f^{\pm}_{*}$ are just the identity map and the deck transformation $\tau$. We call the associated spin structures over $N$ as the \emph{product spin structure} and the \emph{twisted spin structure} respectively. Let $\mathfrak{s}$ be the unique spin structure on $M$. Then we use $\tilde{\mathfrak{s}}$ to denote the former and use $\tilde{\mathfrak{s}}^{\tau}$ to denote the latter. 
\end{defi}
For general $M$, the product family and the twisted family are not isomorphic. For example, Kronheimer-Mrowka \cite{KronheimerMrowkaDehnTwist} established the following example: 
\begin{ex}
The product family $(K3\times S^{1},\tilde{\mathfrak{s}})$ and the twisted family $(K3\times S^{1},\tilde{\mathfrak{s}}^{\tau})$ are not isomorphic, as can be proved by the nonequivariant Bauer-Furuta invariant.  
\end{ex}
However, for the special case of $S^{2}\times S^{2}$, these two families are indeed isomorphic:
\begin{lem}\label{lem: twisted family for S2 cross S2}
The product family $((S^{2}\times S^{2})\times S^{1},\tilde{\mathfrak{s}})$ and the twisted family  $((S^{2}\times S^{2})\times S^{1},\tilde{\mathfrak{s}}^{\tau})$ are isomorphic.
\end{lem}
\begin{proof}
There is an $S^{1}$-action on $S^{2}$ with two fixed points $\{0,\infty\}$. We use $\xi:S^{1}\times S^{2}\rightarrow S^{2}$ to denote this action. As $x$ varies from $0$ to $2\pi$, the induced map
$$
(\text{id}_{S^{2}}\times \xi(x,\cdot))_{*}: T_{(0,0)}(S^{2}\times S^{2})\rightarrow T_{(0,0)}(S^{2}\times S^{2})
$$
gives an essential loop in $SO(4)$. Using this fact, one can verify that the the bundle automorphism $$f:(S^{2}\times S^{2})\times S^{1}\rightarrow (S^{2}\times S^{2})\times S^{1}$$
defined by $f(y_{1},y_{2},x)=(y_{1},\xi(x,y_{2}),x)$ satisfies $f^*(\tilde{\mathfrak{s}})=\tilde{\mathfrak{s}}^{\tau}$.
\end{proof}

\subsubsection{Definition of the Bauer-Furuta invariant}
As in the case of a single 4-manifold, a spin structure  $\mathfrak{s}$ gives rise to two quaternion bundles $S^{\pm}$ over $N$. Denote by $S^{\pm}_{x}$ the restriction of $S^\pm$ to the fiber $M_{x}$. Then the spin Dirac operator  
$$
D(M_{x}):\Gamma(S^{+}_x)\rightarrow \Gamma(S^{-}_{x})
$$
is a quaternionic linear operator. We form the operator $D$ over $N$ by putting $D(M_{x})$ together.

Now we consider four Hilbert bundles $\mathcal{V}^{+}, \mathcal{V}^{-}, \mathcal{U}^{+}$ and $ \mathcal{U}^{-}$  over $B$.
The fibers of $\mathcal{V}^{\pm}$ are suitable Sobolev completions of $\Gamma(S^{\pm}_{x})$. And the fibers of $\mathcal{U}^{+}$ and $ \mathcal{U}^{-}$ are completions of $\Omega^{1}(M_{x})$ and $\Omega^{2}_{+}(M_{x})\oplus \Omega^{0}(M_{x})/\mathbb{R}$ respectively. We let $G=\Pin$ acts on  $\mathcal{V}^{\pm}$ by the left multiplication in the quaternion, and we let $G$ acts on  $ \mathcal{U}^{\pm}$ by setting the $S^{1}$-action to be trivial and setting the $j$-action as multiplication by $-1$.

The family Seiberg-Witten equations give a fiber preserving, $G$-equivariant map 
$$
\mathcal{SW}: \mathcal{U}^{+}\oplus \mathcal{V}^{+}\rightarrow \mathcal{U}^{-}\oplus \mathcal{V}^{-}
$$
This Seiberg-Witten map can be written as $l+c$, where $l$ is the fiberwise Fredholm operator $$l:=D\oplus (d^{+},d^{*})$$
and $c$ is a certain 0-th order operator. Furthermore, by the boundededness property of the Seiberg-Witten equations \cite[Proposition 3.1]{BauerFuruta1}, $\mathcal{SW}$ extends to a map 
$$
\mathcal{SW}^{+}: (\mathcal{U}^{+}\oplus \mathcal{V}^{+})_{\infty}\rightarrow (\mathcal{U}^{-}\oplus \mathcal{V}^{-})_{\infty}
$$
between the one-point completions $$(\mathcal{U}^{\pm}\oplus \mathcal{V}^{\pm})_{\infty}:=(\mathcal{U}^{\pm}\oplus \mathcal{V}^{\pm})\cup \{\infty\}.$$To apply the finite dimensional approximation technique on the map $
\mathcal{SW}$, we carefully choose finite dimensional subspaces of $\mathcal{V}^{\pm}$ and $\mathcal{U}^{\pm}$ as follows: First, we apply Kuiper's theorem \cite{Kuiper} to get canonical trivialization of the bundles
\begin{equation}\label{eq trivilization}
    \mathcal{V}^{-}\cong B\times L^{2}(\mathbb{H}^{\infty})\text{ and }\mathcal{U}^{+}\cong B\times L^{2}(\tilde{\mathbb{R}}^{\infty}).
\end{equation}
Here $L^{2}(*)$ denotes the completion with respect to the $L^{2}$-norm. Choose $m,n\gg 0$ and let $U^{+}\subset \mathcal{U}^{+}$ and $V^{-}\subset \mathcal{V}^{-}$ be the subbundles corresponding to the bundle $B\times \mathbb{H}^{n}$ and $B\times \tilde{\mathbb{R}}^{m}$ under the isomorphism (\ref{eq trivilization}). Let $H^+_{2}$ be the subbundle of $\mathcal{U}^{-}$ consisting of all self-dual harmonic 2 forms on $M_{x}$. We set
$$U^{-}:=H^{+}_{2}\oplus ((d^{+},d^{*})U^{+})\subset \mathcal{U}^{-}.$$
(Note that $(d^{+},d^{*})$ is injective by our assumption $b_{1}(M)=0$.) We choose $m$ large enough so that $V^{-}$ is fiberwise transverse to $D$ and we set $V^{+}:=D^{-1}(V^{-})\subset \mathcal{V}^{+}.$

Set $W^{+}:=U^{+}\oplus V^{+}$ and $W^{-}:=U^{-}\oplus V^{-}$. As explained \cite{BauerFuruta1}, when $m,n$ are large enough, one has 
$$
\mathcal{SW}^{+}(W^{+}_{\infty})\cap S(W^{-,\perp})=\emptyset,
$$
where $S(W^{-,\perp})$ denotes the unit sphere in the orthogonal complement of $W^{-}$ in $\mathcal{U}^{-}\oplus \mathcal{V}^{-}$. Therefore, by composing $\mathcal{SW}^{+}$ with a specific $G$-equivariant deformation retraction 
$$\rho: (\mathcal{U}^{-}\oplus \mathcal{V}^{-})_{\infty}\setminus S(W^{-,\perp})\rightarrow W^{-}_{\infty},$$ one obtains a $G$-equivariant map
$$
sw: W^{+}_{\infty}\rightarrow W^{-}_{\infty} 
$$Restriction of (\ref{eq trivilization}) gives canonical trivializations of the bundles  $V^{-}$ and $U^{+}$. By Assumption \ref{assum: bundle}, $\pi_{1}(B)$ acts trivially on $H^{2}(M_{x})$. Therefore, as explained in \cite{KronheimerMrowkaDehnTwist}, a homology orientation of $M$ determines a canonical trivialization of $H_{2}^{+}$. At this point, we have obtained canonical trivializations of $U^{\pm}$ and $V^{-}$. Using these trivializations, we get the following composition map 
\begin{equation}\label{eq: approximated sw}
(S^{m\tilde{\mathbb{R}}}\wedge V_{\infty}^{+})\cong  W^{+}_{\infty}\xrightarrow{sw} W^{-}_{\infty}\cong (S^{(m+b^{+}(M))\tilde{\mathbb{R}}+n\mathbb{H}}\wedge B_{+}  ) \xrightarrow{pj }S^{(m+b^{+}(M))\tilde{\mathbb{R}}+n\mathbb{H}},
\end{equation}
where $pj$ denotes projection to the first factor. 

From now on, we specialize to the case that $B$ is a circle or point. Note that $V^{+}$ is a quaternionic bundle of dimension $(n-\frac{\sigma(M)}{16})$ and the group
$
Sp(n-\frac{\sigma(M)}{16} )
$ has trivial $\pi_{i}$ for $i\leq 2$. So the bundle $V^{+}$ has a trivialization (canonical up to homotopy). This trivialization allows us to fix an identification $$V^{+}_{\infty}\cong (S^{(n-\frac{\sigma(M)}{16})\mathbb{H} }\wedge B_{+}) $$  
and rewrite the map $(\ref{eq: approximated sw})$ as a $G$-map
\begin{equation}\label{eq: approximated SW map trivilized}
\widetilde{sw}:S^{m\tilde{\mathbb{R}}+(n-\frac{\sigma(M)}{16})\mathbb{H} }\wedge B_{+}\rightarrow S^{(m+b^{+}(M))\tilde{\mathbb{R}}+n\mathbb{H}},
\end{equation}
which represents an element in $[\widetilde{sw}]\in \{S^{-\frac{\sigma(M)}{16}\mathbb{H}}\wedge B_{+},S^{b^{+}(M)\tilde{\mathbb{R}}}\}^{G}$. By checking the  concrete construction of $\widetilde{sw}$ in \cite{BauerFuruta1}, one establishes the following fact: 

\begin{fact}\label{fact: restriction to S1 fixed points}
Consider the map $S^{m\tilde{\mathbb{R}}}\wedge B_{+}\rightarrow S^{(m+b^{+}(M))\tilde{\mathbb{R}}}$ given by restricting $\widetilde{sw}$ to the $S^{1}$-fixed point sets. This map can be explicitly described as the composition
$$
S^{m\tilde{\mathbb{R}}}\wedge B_{+}\xrightarrow{\text{projection}}S^{m\tilde{\mathbb{R}}}\xrightarrow{\text{inclusion} } S^{(m+b^{+}(M))\tilde{\mathbb{R}}}.
$$
\end{fact}

\begin{defi} 
Suppose $B$ is a point. Then $M=N$ and we have $S^{-\frac{\sigma(M)}{16}\mathbb{H}}\wedge B_{+}=S^{-\frac{\sigma(M)}{16}\mathbb{H}}$. In this case, we define the $G$-equivariant Bauer-Furuta invariant as
$$\operatorname{BF}^{G}(M,\mathfrak{s}):=[\widetilde{sw}]\in \{S^{-\frac{\sigma(M)}{16}\mathbb{H}},S^{b^{+}(M)\tilde{\mathbb{R}}}\}^{G}. $$
 We will neglect the spin structure $\mathfrak{s}$ in our notation when it is obvious from the context.
\end{defi}
\begin{ex}\label{ex: S4}
$\BF^{G}(S^{4})$ is an element in $\{S^{0},S^{0}\}^{G}$ represented by a $G$-map from the $S^{m\tilde{\mathbb{R}}+n\mathbb{H}}$ to itself. By the equivariant Hopf theorem \cite[Chapter II.4 ]{tom2011transformation}, such stable homotopy class is determined by its restriction to the $S^{1}$-fixed points. Hence by Fact \ref{fact: restriction to S1 fixed points}, we see that $\BF^{G}(S^{4})=1$. 
\end{ex}
\begin{ex}\label{ex: S2 cross S2}
$\BF^{G}(S^{2}\times S^{2})\in \{S^{0},S^{\tilde{\mathbb{R}}}\}^{G}$ is represented by a $G$-map from $S^{m\tilde{\mathbb{R}}+n\mathbb{H}}$ to $S^{(m+1)\tilde{\mathbb{R}}+n\mathbb{H}}$. Such map are also determined by its restriction on the $S^{1}$-fixed points. By Fact \ref{fact: restriction to S1 fixed points} again, we see that $\BF^{G}(S^{2}\times S^{2})=e_{\tilde{\mathbb{R}}}$. Here $e_{\tilde{\mathbb{R}}}$ is the Euler class defined in (\ref{eq: euler class})
\end{ex}
When $B$ is a circle, we identify it with the unit sphere $S(2\mathbb{R})$ in $S^{2\mathbb{R}}$. 
Consider the cofiber sequence
\begin{equation}\label{eq: cofiber sq 3}
S(2\mathbb{R})\cup \{\infty\}\rightarrow S^{0}\rightarrow  S^{2\mathbb{R}}\xrightarrow{p} S^{\mathbb{R}}\wedge (S(2\mathbb{R})\cup \{\infty\}).
\end{equation}
The map $p$, which is just the Pontryagin-Thom map for the inclusion $S(2\mathbb{R})\hookrightarrow S^{2\mathbb{R}}$, can be treated as a stable map from $S^{\mathbb{R}}$ to $B_{+}$. This stable map induces the map
$$
p^{*}: \{S^{-\frac{\sigma(M)}{16}\mathbb{H}}\wedge B_{+},S^{b^{+}(M)\tilde{\mathbb{R}}}\}^{G}\rightarrow \{S^{\mathbb{R}-\frac{\sigma(M)}{16}\mathbb{H}},S^{b^{+}(M)\tilde{\mathbb{R}}}\}^{G}
$$
that sends $\alpha$ to $\alpha\circ (\text{id}_{ S^{-\frac{\sigma(M)}{16}\mathbb{H}}}\wedge p)$.
\begin{defi}
When $B=S(2\mathbb{R})$, we define the $G$-equivariant Bauer-Furuta invariant as $$\operatorname{BF}^{G}(N,\mathfrak{s}):=p^{*}[\widetilde{sw}]\in\{S^{\mathbb{R}-\frac{\sigma(M)}{16}\mathbb{H}},S^{b^{+}(M)\tilde{\mathbb{R}}}\}^{G}.$$ 
\end{defi}
In both cases, we define the $S^{1}$-equivariant Bauer-Furuta invariant and the nonequivariant Bauer-Furuta invariant as the restriction of the $G$-equivariant Bauer-Furuta invariant: 
$$
\BF^{S^{1}}(N,\mathfrak{s}):=\Res^{G}_{S^{1}}(\BF^{G}(N,\mathfrak{s}))$$
$$\BF^{\{e\}}(N,\mathfrak{s}):=\Res^{G}_{\{e\}}(\BF^{G}(N,\mathfrak{s}))
.$$
In \cite{KronheimerMrowkaDehnTwist}, Kronheimer-Mrowka gave an alternative definition of  $\BF^{\{e\}}(N,\mathfrak{s})$: Take a generic section $r$ of the bundle $W^{-}$ that is transverse to the map $sw$. Then the preimage $sw^{-1}(r)$ is a manifold. When $B$ is a point, the canonical trivilizations of the bundles $W^{\pm}$ determine a stable framing on $sw^{-1}(r)$. When $B$ is $S(2\mathbb{R})$, we fix a stable framing on $B$ that bounds a framed disk. Then together with the trivilizations of $W^{\pm}$, this determines a stable framing on $sw^{-1}(r)$. In \cite{KronheimerMrowkaDehnTwist}, the family Bauer-Furuta invariant is defined as the framed cobordism class of $sw^{-1}(r)$.

Recall that the framed cobordism classes of smooth $n$-manifolds are classified by elements in the $n$-th stable homotopy group of spheres. The following lemma states that our definition of $\BF^{\{e\}}$ is essentially  identical to Kronheimer-Mrowka's definition.

\begin{lem}\label{lem: equivalent definition of BF}
The framed cobordism class of $sw^{-1}(r)$ is classified by the nonequivariant Bauer-Furuta invariant $\BF^{\{e\}}(N,\mathfrak{s})$.
\end{lem}
\begin{proof}
By Sard's theorem, we can take $r$ to be a constant section that sends the whole $B$ to a generic point $r_{0}\in S^{(m+b^{+}(M))\tilde{\mathbb{R}}+n\mathbb{H}}$. Then $sw^{-1}(r)=\widetilde{sw}^{-1}(r_{0})$ and it is also the preimage of the point $$\{0\}\times r_{0}\in S^{\mathbb{R}+(m+b^{+}(M))\tilde{\mathbb{R}}+n\mathbb{H}} $$ under the composition 
\begin{equation}\label{eq: Bauer-Furuta unstable map}
( \text{id}_{S^{\tilde{\mathbb{R}}}}\wedge \widetilde{sw})\circ (\text{id}_{ S^{(n-\frac{\sigma(M)}{16})\mathbb{H}+m\tilde{\mathbb{R}}}}\wedge p):S^{2\mathbb{R}+m\tilde{\mathbb{R}}+(n-\frac{\sigma(M)}{16})\mathbb{H} }\rightarrow S^{\mathbb{R}+(m+b^{+}(M))\tilde{\mathbb{R}}+n\mathbb{H}}.
\end{equation}

Because $r_{0}$ is a regular value of $\widetilde{sw}$ and any point in $\{0\}\times B_{+}$ is a regular value of $p$, we see that $\{0\}\times r_{0}$ is indeed a regular value of the map (\ref{eq: Bauer-Furuta unstable map}). Recall that an element in the stable group of spheres defines a stably framed manifold by taking the preimage of a regular value and taking the induced framing. The proof is finished by observing that the stable framing on $B$ that bounds a framed disk (the one we used to fix the framing on $sw^{-1}(r)$) is exactly the framing induced by the inclusion $B\hookrightarrow S^{2\mathbb{R}}$. 
\end{proof}

\subsubsection{Some properties of the Bauer-Furuta invariant}
In this subsection, we summarize some important properties of the Bauer-Furuta invairant. We start with a vanishing result. Recall from Definition 
\ref{defi: product family and twisted family} that on the trivial bundle $N=M\times S^{1}$, there are two spin structures: the product spin structure $\tilde{\mathfrak{s}}$ and the twisted spin structure $\tilde{\mathfrak{s}}^{\tau}$.  

\begin{lem}\label{lem: vanishing}
The Bauer-Furuta invariants $\BF^{G}$, $\BF^{S^{1}}$ and  $\BF^{\{e\}}$ of the product spin structure $\tilde{\mathfrak{s}}$ are all vanishing.
\end{lem}
\begin{proof}
The cofiber sequence (\ref{eq: cofiber sq 3}) induces a long exact sequence
$$
\cdots \rightarrow \{S^{-\frac{\sigma(M)}{16}\mathbb{H}},S^{b^{+}(M)\tilde{\mathbb{R}}}\}^{G}\xrightarrow{q^{*}}\{S^{-\frac{\sigma(M)}{16}\mathbb{H}}\wedge B_{+},S^{b^{+}(M)\tilde{\mathbb{R}}}\}^{G}\xrightarrow{p^{*}} \{S^{\mathbb{R}-\frac{\sigma(M)}{16}\mathbb{H}},S^{b^{+}(M)\tilde{\mathbb{R}}}\}^{G}\rightarrow \cdots 
$$
where $q^{*}$ is induced by the map $q:B_{+}\rightarrow S^{0}$ that preserves the base point and sends $B$ to the other point. By its definition, the map $\widetilde{sw}$ for $(M\times S^{1},\tilde{\mathfrak{s}})$ is just pull-back of the corresponding map for $(M,\mathfrak{s})$ via the map $q$. So $[\widetilde{sw}]\in \text{Image}( q^{*})$, which implies $$\operatorname{BF}^{G}((M\times S^{1},\tilde{\mathfrak{s}}))=p^{*}( [\widetilde{sw}])=0.$$ The invariants $\BF^{S^{1}}$ and  $\BF^{\{e\}}$ vanish  because  $\BF^{G}$ is vanishing.
\end{proof}
Regarding the Bauer-Furuta invariant of the twisted spin structure, Kronheimer-Mrowka \cite{KronheimerMrowkaDehnTwist} proved the following result by studying the stable framing on the moduli space.
\begin{pro}\label{lem: twisted family} We have 
\begin{equation}
    \BF^{\{e\}}(M\times S^{1},\tilde{\mathfrak{s}}^{\tau})=\begin{cases}
    \eta\cdot \BF^{\{e\}}(M,\mathfrak{s}) &\text{when } \sigma(M)\equiv 16 \mod32\\
    0 &\text{when } 32\mid \sigma(M).
    \end{cases}
\end{equation}
Here $\eta\in \{S^{\mathbb{R}},S^{0}\}^{\{e\}}$ denotes the Hopf map. 
\end{pro}
\begin{rmk}It would be interesting to prove a generalization of Lemma \ref{lem: twisted family} for $\BF^{G}(M\times S^{1},\tilde{\mathfrak{s}}^{\tau})$ and $\BF^{S^{1}}(M\times S^{1},\tilde{\mathfrak{s}}^{\tau})$. 
\end{rmk}

Next, we give a connected sum formula for the family Bauer-Furuta invariants. This formula was orginally proved by Bauer \cite{BauerFurutaII} for a single 4-manifold. 

To set up the theorem, we let $(N_{i},\mathfrak{s}_{i})$ $(i=1,2)$ be two spin families over $B=S(2\mathbb{R})$ with fiber $M_{i}$, both satisfying Assumption \ref{assum: bundle}. To form the connected sum, we pick sections $\gamma_{i}:B\rightarrow N_{i}$. By our Assumption \ref{assum: bundle} (i), the section $\gamma_{i}$ is unique up to homotopy. We remove small, standard 4-balls around these sections to form the family $N_{i}-D^{4}\times S^{1}$ of 4-manifolds with boundary. Then we can form the fiberwise connected sum by identifying the collars of their boundaries. To fix such identification, we need to choose a smooth family of orientation reversing isomorphisms 
$$
\tilde{\phi}:=\{\phi_{x}:T_{\gamma_{1}(x)}(M_{1})_{x}\xrightarrow{\cong} T_{\gamma_{2}(x)}(M_{2})_{x}\}_{x\in B}.
$$
We use $N_{1}\#_{\tilde{\phi}}N_{2}$ to denote the resulting bundle over $B$, with fiber $M_{1}\# M_{2}$. In general, the result  $N_{1}\#_{\tilde{\phi}}N_{2}$ will depend on choices of $\tilde{\phi}$ up to homotopy. Because $\pi_{1}(\SO(4))=\mathbb{Z}/2$, there are essentially two choices.

\begin{lem}\label{lem: gluing spin structure}
There exists exactly one choice of $\widetilde{\phi}$ such that the spin structures $\mathfrak{s}_{1},\mathfrak{s}_{2}$ can be glued together to form a spin structure on $N_{1}\#_{\tilde{\phi}}N_{2}$. We denote this choice as $\tilde{\phi}(\mathfrak{s}_{1},\mathfrak{s}_{2})$ and denote the resulting spin structure by $\mathfrak{s}_{1}\#\mathfrak{s}_{2}$.
\end{lem}
\begin{proof}Denote by $\tilde{\phi}^{\pm}$ the two choices of $\tilde{\phi}$. Then they provide gluing maps 
$$
f^{\pm}: \partial (N_{1}-D^{4}\times S^{1})\rightarrow \partial (N_{2}-D^{4}\times S^{1})
$$
which differ from each other by a Dehn twist on $\partial (N_{2}-D^{4}\times S^{1})$. Under any boundary parametrization $\partial (N_{2}-D^{4}\times S^{1})\cong S^{3}\times S^{1}$, this Dehn twist can be written as $$\iota(v,x)=(\alpha(x)v,x) \text{ for }(v,x)\in S^{3}\times S^{1},$$ where $\alpha:S^{1}\rightarrow SO(4)$ is an essential loop. Note that $S^{3}\times S^{1}$, regarded as the product $S^{3}$-bundle over $S^{1}$, has two family spin structures (the product spin structure and the twisted spin structure), which are related to each other by $\iota$. We see that exactly only of the two maps $f^{\pm}$ sends $\mathfrak{s}_{1}|_{\partial (N_{1}-D^{4}\times S^{1}) }$ to   $\mathfrak{s}_{2}|_{\partial (N_{2}-D^{4}\times S^{1}) }$. This finishes the proof. We also note that when $\tilde{\phi}=\tilde{\phi}(\mathfrak{s}_{1},\mathfrak{s}_{2}),$ the gluing map on the boundary has two lifts to the gluing map on the spin bundle, but they give isomorphic spin structures on the connected sum.
\end{proof}
From the discussion above, there is a unique way to take connected sum of two spin families $(N_{i},\mathfrak{s}_{i})$ together. The resulting spin family 
$(N_{1}\#_{\widetilde{\phi}(\mathfrak{s}_{1}, \mathfrak{s}_{2})}N_{2}, \mathfrak{s}_{1}\# \mathfrak{s}_{2}) $
will also  be written as $(N_{1},\mathfrak{s}_{1})\#(N_{2},\mathfrak{s}_{2})$. 

To talk about the Bauer-Furuta invariant of a connected sum, we also need to specify a rule for homology orientation as follows: Given homology orientations on $M_{1},M_{2}$, we let the homology orientation on $M_{1}\sharp M_{2}$ be defined by putting the oriented basis for $H_{+}^{2}(M_{1})$ in front of the oriented basis for $H_{+}^{2}(M_{2})$. The following theorem is a family version of Bauer's connected sum formula \cite{BauerFurutaII}.

\begin{pro}\label{pro: connected sum}
Let $(M\times S^{1},\widetilde{\mathfrak{s}})$ be the product family for some spin 4-manifold $(M, \mathfrak{s})$. Then we have  $$\operatorname{BF^{H}}((N_{1},\mathfrak{s}_{1})\#(M\times S^{1},\widetilde{\mathfrak{s}}))=\BF^{H}(N_{1},\mathfrak{s}_{1})\wedge \BF^{H}(M,\mathfrak{s})$$ for $H=G,S^{1}$ or $\{e\}$.
\end{pro}

\begin{proof}
The proof is essentially identical with the single 4-manifold case in \cite{BauerFurutaII}. (See \cite{KronheimerMrowkaDehnTwist} for a sketch proof for the family version (in the nonequivariant setting). A central step is an excision argument that builds a homotopy between the approximated Seiberg-Witten maps $\widetilde{sw}$ (\ref{eq: approximated SW map trivilized}) for the bundle
$$
N_{1}\cup (M\times S^{1})\cup (S^{4}\times S^{1}),
$$
viewed as a family over $S^{1}$ with fiber $M_{1}\cup M\cup S^{4}$ and the bundle $$(N\#(M\times S^{1})\cup (S^{4}\times S^{1}))\cup (S^{4}\times S^{1}),$$ viewed as a family over $S^{1}$ with fiber $(M_{1}\#M)\cup S^{4}\cup S^{4}$. This homotopy is constructed by multiplying various sections by scalar-valued real cut-off functions and applying various terms in the Seiberg-Witten map, which are all $G$-equivariant. Therefore, this homotopy is $G$-equivariant.
\end{proof}

As a corollary, we get the following result that computes the Bauer-Furuta invariant under family stabilization: 

\begin{cor}\label{cor: stablization BF}
Consider the product spin structure $\tilde{\mathfrak{s}}_{0}$ and the twisted spin structure $\tilde{\mathfrak{s}}^{\tau}_{0}$ over the product bundle $((S^{2}\times S^{2})\times S^{1})$. Then for any spin family $(N,\mathfrak{s})$ that satisfies Assumption \ref{assum: bundle},  we have 
\begin{equation}\label{eq: stablization product}
\BF^{G}((N,\mathfrak{s})\#(((S^{2}\times S^{2})\times S^{1}),\tilde{\mathfrak{s}}_{0}))=\BF^{G}(N,\mathfrak{s})\cdot e_{\tilde{\mathbb{R}}},
\end{equation}
and 
\begin{equation}\label{eq: stablization twist}
\BF^{G}((N,\mathfrak{s})\#(((S^{2}\times S^{2})\times S^{1}),\tilde{\mathfrak{s}}^{\tau}_{0}))=\BF^{G}(N,\mathfrak{s})\cdot e_{\tilde{\mathbb{R}}}.
\end{equation}
Here $e_{\tilde{\mathbb{R}}}\in \{S^{0},S^{\tilde{\mathbb{R}}}\}^{G}$ is the Euler class defined in (\ref{eq: euler class}).
\end{cor}
\begin{proof}
The formula (\ref{eq: stablization product}) follows from Proposition \ref{pro: connected sum}, Example \ref{ex: S2 cross S2}. The formula (\ref{eq: stablization twist}) follows from (\ref{eq: stablization product}) and Lemma \ref{lem: twisted family for S2 cross S2}. 
\end{proof}

\section{Proof of the Main Theorem}
\subsection{The Key Proposition}
In this subsection, we prove a homotopy theoretic proposition (Proposition \ref{pro: key}), which will be the key ingredient in the proof of our main theorem.

Recall that the group $\{S^{\mathbb{R}+2\mathbb{H}},S^{6\tilde{\mathbb{R}}}\}^{S^{1}}$ admits a conjugation action $c_{j}$ (see (\ref{conjugation map})). The following lemma computes this group and this action.
\begin{lem}\label{lem: 2H 6R calculation}
The characteristic homomorphism $t:\{S^{\mathbb{R}+2\mathbb{H}},S^{6\tilde{\mathbb{R}}}\}^{S^{1}}\rightarrow \mathbb{Z}$ is surjective and has $\ker t=\mathbb{Z}/2$. The conjugation action $c_{j}$ acts trivially on $\ker t$. 
\end{lem}
\begin{proof}
Smashing the cofiber sequence $S^{0}\rightarrow S^{2\mathbb{H}}\rightarrow S^{\mathbb{R}}\wedge (S(2\mathbb{H})_{+})$ with $S^{\mathbb{R}}$, we get a cofiber sequence $S^{\mathbb{R}}\rightarrow S^{\mathbb{R}+2\mathbb{H}}\rightarrow S^{2\mathbb{R}}\wedge (S(2\mathbb{H})_{+}),$ which induces the long exact sequence
$$
\cdots \rightarrow \{S^{2\mathbb{R}},S^{6\tilde{\mathbb{R}}}\}^{S^{1}}\rightarrow \{S^{2\mathbb{R}}\wedge (S(2\mathbb{H})_{+}) ,S^{6\tilde{\mathbb{R}}}\}^{S^{1}}\rightarrow \{S^{\mathbb{R}+2\mathbb{H}},S^{6\tilde{\mathbb{R}}}\}^{S^{1}}\rightarrow \{S^{\mathbb{R}},S^{6\tilde{\mathbb{R}}}\}^{S^{1}}\rightarrow
$$
By the equivariant Hopf theorem \cite[Chapter II.4 ]{tom2011transformation}, we have $\{S^{\mathbb{R}},S^{6\tilde{\mathbb{R}}}\}^{S^{1}}=\{S^{2\mathbb{R}},S^{6\tilde{\mathbb{R}}}\}^{S^{1}}=0$. Hence we get the isomorphism 
$$\{S^{\mathbb{R}+2\mathbb{H}},S^{6\tilde{\mathbb{R}}}\}^{S^{1}}\cong \{S^{2\mathbb{R}}\wedge (S(2\mathbb{H})_{+}) ,S^{6\tilde{\mathbb{R}}}\}^{S^{1}}.$$ Note that the $S^{1}$-action on $S^{2\mathbb{R}}\wedge (S(2\mathbb{H})_{+})$ is free way from base point. By Fact \ref{fact: quotient isomorphism}, we have 
$$\{S^{2\mathbb{R}}\wedge (S(2\mathbb{H})_{+}) ,S^{6\tilde{\mathbb{R}}}\}^{S^{1}}=\{S^{2\mathbb{R}}\wedge(\mathbb{C}P^{3}_+) ,S^{6\tilde{\mathbb{R}}} \}^{\{e\}}.$$

The cofiber sequence $\mathbb{C}P^{1}_{+}\rightarrow \mathbb{C}P^{3}_{+}\rightarrow  \mathbb{C}P^{3}/\mathbb{C}P^{1}$
induces the exact sequence 
$$
\{S^{3\mathbb{R}}\wedge(\mathbb{C}P^{1}_{+}), S^{6\tilde{\mathbb{R}}}\}^{\{e\}}\rightarrow \{ S^{2\mathbb{R}}\wedge(\mathbb{C}P^{3}_+) , S^{6\tilde{\mathbb{R}}}\}^{\{e\}}\rightarrow\{ S^{2\mathbb{R}}\wedge(\mathbb{C}P^{3}/\mathbb{C}P^{1}) , S^{6\tilde{\mathbb{R}}}\}^{\{e\}}\rightarrow \{S^{2\mathbb{R}}\wedge(\mathbb{C}P^{1}_{+}), S^{6\tilde{\mathbb{R}}}\}^{\{e\}}.
$$
By the cellular approximation theorem, we have $$\{S^{3\mathbb{R}}\wedge(\mathbb{C}P^{1}_{+}), S^{6\tilde{\mathbb{R}}}\}^{\{e\}}=\{S^{2\mathbb{R}}\wedge(\mathbb{C}P^{1}_{+}), S^{6\tilde{\mathbb{R}}}\}^{\{e\}}=0.$$ 
So we obtain the isomorphism $$ \{S^{2\mathbb{R}}\wedge(\mathbb{C}P^{3}_+) , S^{6\tilde{\mathbb{R}}}\}^{\{e\}}\cong\{ S^{2\mathbb{R}}\wedge(\mathbb{C}P^{3}/\mathbb{C}P^{1}) , S^{6\tilde{\mathbb{R}}}\}^{\{e\}}.$$ 

To understand the stable homotopy type of $\mathbb{C}P^{3}/\mathbb{C}P^{1}$ as a nonequivariant space, we let $x$ be the generator of $H^{2}(\mathbb{C}P^{3};\mathbb{Z}/2)$. Then the total Steenrod square is given by 
$$\operatorname{Sq}(x)=\operatorname{Sq}^{0}(x)+\operatorname{Sq}^2(x)=x+x^{2}.$$ By the Cartan formula, we get
$$
\operatorname{Sq}(x^2)= (x+x^{2})^2=x^{2}\in H^{*}(\mathbb{C}P^{3};\mathbb{Z}/2).
$$ 
In particular, we get $\operatorname{Sq}^2(x^2)=0$, which implies that the attaching map between the 6-cell and the 4-cell in $\mathbb{C}P^{3}$, regarded as an element in the stable homotopy group $\pi_{1}=\mathbb{Z}/2$, is trivial. Therefore, we conclude that $\mathbb{C}P^{3}/\mathbb{C}P^{1}$ is stably homotopy equivalent to $S^{6\mathbb{R}}\vee S^{4\mathbb{R}}$. This implies $$\{ S^{2\mathbb{R}}\wedge(\mathbb{C}P^{3}/\mathbb{C}P^{1}) , S^{6\tilde{\mathbb{R}}}\}^{\{e\}}=\pi_{2}\oplus \pi_{0}=\mathbb{Z}/2\oplus \mathbb{Z}.$$ The projection to the $\pi_{0}$-summand can be alternatively defined as the mapping degree on $H^{6}(*;\mathbb{Z})$, so it is exactly the characteristic homomorphism $t$. We have shown that $t$ is surjective with kernel $\mathbb{Z}/2$.  By Corollary \ref{cor: conjugate char homo}, we have $t(c_{j}(\alpha))=t(\alpha)$ for any $\alpha\in \{S^{\mathbb{R}+2\mathbb{H}},S^{6\tilde{\mathbb{R}}}\}^{S^{1}}$. So $c_{j}$ must send $\ker t$ to $\ker t$. Since $\ker t\cong \mathbb{Z}/2$, $c_{j}$ must act trivially on it.
\end{proof}

\begin{pro}\label{pro: key}
Let $\alpha$ be an element in $\{S^{\mathbb{R}+2\mathbb{H}},S^{6\tilde{\mathbb{R}}}\}^{G}$ that satisfies the conditions 
$$t(\Res^{G}_{S^{1}}(\alpha))=0,\text{ and } \alpha\cdot e_{\tilde{\mathbb{R}}}=0.$$ Then $\Res^{G}_{S^{1}}(\alpha)=0$.
\end{pro}
\begin{proof}
By Lemma \ref{lem: image of transfer}, we see that $\alpha=\Tr^{G}_{S^{1}}(\beta)$ for some $\beta\in \{S^{\mathbb{R}+2\mathbb{H}},S^{6\tilde{\mathbb{R}}}\}^{S^{1}}$. Therefore, by the double coset formula (\ref{eq: double coset formula}), we have $\Res^{G}_{S^{1}}(\alpha)=\beta+c_{j}(\beta)$. By lemma \ref{cor: conjugate char homo}, we get $$0=t(\beta+c_{j}(\beta))=2t(\beta).$$ So $\beta$ is in the kernel of $t$, which is $\mathbb{Z}/2$ by Lemma \ref{lem: 2H 6R calculation}. By Lemma \ref{lem: 2H 6R calculation} again, we have $c_{j}(\beta)=\beta$. So we get $\Res^{G}_{S^{1}}(\alpha)=2\beta=0$.
\end{proof}
\subsection{Proof of Theorem \ref{thm: main} }
Let $X_{1}$ be the $K3$ surface and $X_{0}=S^{2}\times S^{2}$. Let $\mathfrak{s}_{i}$ be the unique spin structure on $X_{i}$ for $i=0,1$.  
We consider the Dehn twist $$\delta: X_{1}\# X_{1}\rightarrow X_{1}\# X_{1}$$ 
along the separating $S^{3}$ in the neck. We want to show that $\delta$ is not smoothly isotopic to the identity map even after a single stabilization. Without loss of generality, we may assume that the stabilization is done in the first copy of $X_{1}$. Then we need to show that the map
$$
\delta^{s}:=\text{id}_{X_{0}}\#\delta : X_{0}\#X_{1}\# X_{1}\rightarrow X_{0}\#X_{1}\# X_{1}
$$
is not smoothly isotopic to the identity map. As in \cite{KronheimerMrowkaDehnTwist}, we will prove this by forming the mapping torus 
$$
N_{\delta^{s}}:=((X_{0}\#X_{1}\# X_{1})\times [0,1])/(x,0)\sim (\delta^{s}(x),1)
$$
and show that it a nontrivial smooth bundle over $S^{1}$. 

By Lemma \ref{lem: vanishing}, the product spin structure over the trivial bundle has vanishing $\BF^{G}$. Therefore, it suffices to show that both spin families associated to $N_{\delta^{s}}$ has nontrivial $\BF^{G}$. 

To prove this, we consider the product family $(X_{i}\times S^{1}, \tilde{\mathfrak{s}}_{i})$ and the twisted family $(X_{i}\times S^{1}, \tilde{\mathfrak{s}}^{\tau}_{i})$. By the discussion in \cite[Begining of Section 5]{KronheimerMrowkaDehnTwist}, the mapping torus $N_{\delta}$ can be formed as the fiberwise connected sum 
$$
(X_{1}\times S^{1})\#_{\varphi(\tilde{\mathfrak{s}}_{1},\tilde{\mathfrak{s}}^{\tau}_{1})}(X_{1}\times S^{1}).$$
Therefore, the bundle $N_{\delta^{s}}$ can formed as the fiberwise connected sum 
$$
(X_{0}\times S^{1})\#_{\varphi(\tilde{\mathfrak{s}}_{0},\tilde{\mathfrak{s}}_{1})}(X_{1}\times S^{1})\#_{\varphi(\tilde{\mathfrak{s}}_{1},\tilde{\mathfrak{s}}^{\tau}_{1})}(X_{1}\times S^{1})
$$
as well as the fiberwise connected sum 
 $$
(X_{0}\times S^{1})\#_{\varphi(\tilde{\mathfrak{s}}^{\tau}_{0},\tilde{\mathfrak{s}}^{\tau}_{1})}(X_{1}\times S^{1})\#_{\varphi(\tilde{\mathfrak{s}}_{1}^{\tau},\tilde{\mathfrak{s}}_{1})}(X_{1}\times S^{1}).
$$
The two spin families associated to $N_{\delta^{s}}$ are  $$(X_{0}\times S^{1}, \tilde{\mathfrak{s}}_{0})\#(X_{1}\times S^{1}, \tilde{\mathfrak{s}}_{1})\#(X_{1}\times S^{1}, \tilde{\mathfrak{s}}^{\tau}_{1}) $$
and 
$$(X_{0}\times S^{1}, \tilde{\mathfrak{s}}^{\tau}_{0})\#(X_{1}\times S^{1}, \tilde{\mathfrak{s}}^{\tau}_{1})\#(X_{1}\times S^{1}, \tilde{\mathfrak{s}}_{1}) .$$
We will show that $$\BF^{G}((X_{0}\times S^{1}, \tilde{\mathfrak{s}}_{0})\#(X_{1}\times S^{1}, \tilde{\mathfrak{s}}_{1})\#(X_{1}\times S^{1}, \tilde{\mathfrak{s}}^{\tau}_{1}))\neq 0$$ and the other family is similar. We use $\alpha$ to denote the element  $$\BF^{G}((X_{1}\times S^{1}, \tilde{\mathfrak{s}}_{1})\#(X_{1}\times S^{1}, \tilde{\mathfrak{s}}^{\tau}_{1}))\in \{S^{\mathbb{R}+2\mathbb{H}},S^{6\tilde{\mathbb{R}}}\}^{G}.$$
By Proposition \ref{pro: connected sum}, $\Res^{G}_{S^{1}}(\alpha)$ can be decomposed as the product of the elements 
$$
\BF^{S^{1}}(X_{1}, \mathfrak{s}_{1})\in \{S^{\mathbb{H}},S^{3\tilde{\mathbb{R}}}\}^{S^{1}}\text{ and }\BF^{S^{1}}((X_{1}\times S^{1}, \tilde{\mathfrak{s}}^{\tau}_{1}))\in \{S^{\mathbb{R}+\mathbb{H}},S^{3\tilde{\mathbb{R}}}\}^{S^1}.
$$
By Lemma \ref{lem: SW vanishing}, the Seiberg-Witten invariant $t(\Res^{G}_{S^{1}}(\alpha))=0$. (This can also be proved by checking the explicit description of the Seiberg-Witten moduli space given in \cite{KronheimerMrowkaDehnTwist}.) 

By Corollary \ref{cor: stablization BF}, we have $$\BF^{G}((X_{0}\times S^{1}, \tilde{\mathfrak{s}}_{0})\#(X_{1}\times S^{1}, \tilde{\mathfrak{s}}_{1})\#(X_{1}\times S^{1}, \tilde{\mathfrak{s}}^{\tau}_{1}))=\alpha\cdot e_{\tilde{\mathbb{R}}} $$
For the sake of contradiction, we suppose $\alpha\cdot e_{\tilde{\mathbb{R}}}=0$ . Then by Proposition \ref{pro: key}, we have $\Res^{G}_{S^{1}}(\alpha)=0$, which implies 
$$
\BF^{\{e\}}((X_{1}\times S^{1}, \tilde{\mathfrak{s}}_{1})\#(X_{1}\times S^{1}, \tilde{\mathfrak{s}}^{\tau}_{1}))=\Res^{G}_{\{e\}}(\alpha)= \Res^{S^{1}}_{\{e\}}\circ \Res^{G}_{S^{1}}(\alpha)=0. 
$$
However, Kronheimer-Mrowka \cite[Proposition 5.1]{KronheimerMrowkaDehnTwist} computed this nonequivariant Bauer-Furuta invariant as $\eta^{3}\neq 0\in \pi_{3}$. (Kronheimer-Mrowka's definition of $\BF^{\{e\}}$ coincides with ours because of Lemma \ref{lem: equivalent definition of BF}.) This is a contradiction and our proof is finished.

\bibliographystyle{amsplain}
\bibliography{Bbib0320}
\end{document}